\documentclass[12pt,reqno]{amsart}
\usepackage{amsfonts}
\usepackage{tipa}
\usepackage{amssymb}
\usepackage{mathrsfs}
\usepackage{amsmath}
\usepackage{txfonts}
\usepackage{amsthm}
\usepackage{cite}
\theoremstyle{plain}
\newtheorem{Thm}{Theorem}[section]

\newtheorem{Lem}[Thm]{Lemma}

\newtheorem{Rk}[Thm]{Remark}

\numberwithin{equation}{section}

\newcommand{\ld}{\lambda}

\begin{document}
\title[nonlocal parabolic equation]{Vacuum isolating, blow up threshold and asymptotic behavior of solutions for a nonlocal parabolic equation}

\author[X. Li]{Xiaoliang Li}
\author[B. Liu]{Baiyu Liu}

\address[X. Li, B. Liu]{School of Mathematics and Physics\\
  University of Science and Technology Beijing \\
  30 Xueyuan Road, Haidian District
  Beijing, 100083\\
  P.R. China}
\email{liuby@ustb.edu.cn, liubymath@gmail.com}


\thanks{$^*$ Project supported by  the National Natural Science Foundation of China No.11671031, No.11201025 .}

\begin{abstract}
In this paper, we consider a nonlocal parabolic equation associated with initial and Dirichlet boundary conditions. Firstly, we discuss the vacuum isolating behavior of solutions with the help of a family of potential wells. Then we obtain a threshold of global existence and blow up for solutions with critical initial energy. Furthermore, for those solutions satisfy $J(u_0)\leq d$ and $I(u_0)\neq 0$, we show that global solutions decay to zero exponentially as time tends to infinity and the norm of blow-up solutions increase exponentially.
\end{abstract}

\maketitle

\section{Introduction}
In this paper, we study the following initial boundary value problem of nonlocal parabolic equation
\begin{equation}
\label{eq:pro}
\left\{
\begin{array}{ll}
u_t=\Delta u+\left(\frac{1}{|x|^{n-2}}*|u|^p\right)|u|^{p-2}u,& x\in \Omega, t>0,\\
u(x,t)=0,& x\in \partial \Omega, t>0,\\
u(x,0)=u_0(x),& x\in \Omega,
\end{array}
\right.
\end{equation}
where $\Omega$ is a bounded domain in $\mathbb{R}^n$  $(n\geq 3)$, $1< p<(n+2)/(n-2)$ and $\frac{1}{|x|^{n-2}}*|u|^p=\int_{\Omega}\frac{|u(y)|^p}{|x-y|^{n-2}}dy$.

Nonlocal parabolic type equations have been extensively used in ecology, especially to model a population in which individual competes for a shared rapidly equilibrated resource or a population in which individual communicated either visually or by chemical means \cite{Furter, Gourley, Ou2007, So2012}. Also, they can be applied to thermal physics with nonlocal source \cite{Lacey}.

As a model problem for studying the competition between the dissipative effect of diffusion and the influence of an explosive source term, problem
\begin{equation}\label{pro:sca}
\left\{
\begin{array}{ll}
u_t=\Delta u+|u|^{p-1}u, & x\in \Omega, t>0\\
u(x,t)=0, & x\in \partial \Omega, t\geq 0\\
u(x,0)=u_0(x), & x\in \Omega
\end{array}
\right.
\end{equation}
has been extensively studied (see \cite{Giga1986CMP, Liuya, Ma2012, IkeSuz1996HMJ, HosYam1991FE, Bal1977QJM, Filippo, Flavio} and the reference therein).
For the sub-critical case $1<p<(n+2)/(n-2)$, blow up in infinite time does not occur. The solution will either exist globally or blow up in finite time.
It is natural to ask under what conditions, will the
solution exist for all time; and under what conditions, will the solution
become unstable to collapse. To treat the above question, Sattinger \cite{Sattinger} (see also \cite{Payne}) established a powerful method which is called the potential well method.  By using this method, Ikehata and Suzuki \cite{IkeSuz1996HMJ}, Payne and Sattinger \cite{Payne} described the behavior of solutions for (\ref{pro:sca}) when the initial data has low energy (smaller than the height of potential well). Roughly speaking, they found a threshold of global solutions and blow up solutions.
Liu and Zhao \cite{Liuya}, Xu \cite{Xu2010QAM}  generalized the above results to the critical energy level initial data.
Moreover, by generalizing the potential well method, an important phenomena called vacuum isolating has been found by Liu and Zhao\cite{Liuya}, i.e., there is a region which does not contain any low energy solutions. Vacuum isolating phenomena has also been observed in various kinds of evolution equations with variational structures \cite{LiuJDE2003, LiuXu2007JMAA, Chenh}.

As a model problem of nonlocal parabolic equation, (\ref{eq:pro}) has been studied by \cite{LiuB2014, LMCPAA}. Well-posedness in $L^q(\Omega)$ has been setup. Precisely,
\begin{Thm}\label{thm:ste}[Theorem 6 and 7 in \cite{LiuB2014}]
	Let $u_0\in L^{q}(\Omega)$, $n-1\leq q<\infty$, $q>\frac{n}{2}(p-1)(2-\frac{1}{p})$. Then there exists $T_{\max}=T(||u_0||_{q})>0$ such that problem (\ref{eq:pro}) possesses a unique classical $L^{q}-$solution in $[0,T_{\max})$.
Moreover, either $T_{\max}=+\infty$ or $\lim_{t\to T_{\max}}||u(t)||_{q}=+\infty$.
\end{Thm}

There are two natural functionals on $H_0^1(\Omega)$ associated with the problem (\ref{eq:pro}), the energy functional and the Nehari functional, defined respectively by
$$
J(u)=\frac{1}{2}\int_\Omega |\nabla u|^2dx-\frac{1}{2p}\int_{\Omega\times \Omega}\frac{|u(y)|^p|u(x)|^p}{|x-y|^{n-2}}dxdy,
$$
$$
I(u(t))=(J'(u),u)=\int_\Omega |\nabla u|^2dx-\int_{\Omega}\left(\frac{1}{|x|^{n-2}}*|u|^p\right)|u|^pdx,
$$

Then along the flow generated by (\ref{eq:pro}), we have
\begin{equation}
	\label{eq:decj}
	\frac{d}{dt}J(u(t))=(J'(u),u_t)=-||u_t||_2^2\leq 0.
\end{equation}

The Nehari manifold is defined by
	\begin{equation}
	\label{def:N}
	N:=\{u\in H_0^1(\Omega)|I(u)=0,u\neq 0\}.
	\end{equation}
The depth of the potential well is
	\begin{equation}
		\label{def:d}
		d:=\inf_{u\in N}\{J(u)\}.
	\end{equation}

By using the potential well method, Liu and Ma  \cite{LiuB2014} proved that 
for low energy solutions ($J(u_0)<d$) the maximum existence time is totally determined  by the Nehari functional $I(u_0)$.
More precisely, if $J(u_0)<d$ and $I(u_0)>0$ then the solution exists globally, if $J(u_0)<d$ and $I(u_0)<0$ then the solution blows up in finite time.

This paper devoted to continue the study of \cite{LiuB2014}. 
The first result of the present paper deals with the solution start with initial data which has low initial energy. We found the vacuum isolating phenomenon, by using the family of potential wells \cite{Liuya,Xu2013JFA}.

Let $\delta >0$. Define
\begin{gather*}
I_\delta (u) :=\delta||\nabla u||^2-\int_\Omega v(u)|u(x)|^pdx ,\\
N_\delta :=\left\{u \in H_0^1(\Omega)|I_\delta (u)=0,||\nabla u|| \neq 0\right\}, \quad d(\delta) =  \inf_{u \in N_\delta}{J(u)}.
\end{gather*}

\begin{Thm}
	\label{thm:vacuum}
	Let $e\in (0,d)$. Suppose  $\delta_1,\delta_2$ are the two roots of $d(\delta)=e$.Then for all solutions of problem (\ref{eq:pro}) with $J(u_0)\leq e$, there is 
	a vacuum region
	$$ U_e=\bigcup_{\delta_1<\delta<\delta_2}N_\delta = \left \{ u\in H_0^1(\Omega) \mid I_\delta(u)=0, u\neq 0,\delta_1<\delta<\delta_2\right \},$$
	such that there is no any solution of problem (\ref{eq:pro}) in $U_e$.
	\end{Thm}

Then we study the critical initial energy case and obtain the threshold just like the low initial energy solution.

 \begin{Thm}
	\label{thm:3.1}
	Let $\Omega$ be a smooth bounded convex domain in  $\mathbb{R}^n \ (n=3\,\textrm{or}\,4)$. Assume $1<p<\frac{n+2}{n-2}$, such that $(p-1)(2-\frac{1}{p})<\frac{4}{n-2}$. If $u_0\in C(\bar \Omega)\cap H_0^1(\Omega)$ and $J(u_0)=d,I(u_0)>0$, then problem (\ref{eq:pro}) admits a global solution $u(t)$ for $0<t<\infty$.
\end{Thm}

 \begin{Thm}
	\label{thm:3.2}
	Let $\Omega$ be a smooth bounded convex domain and $1<p<\frac{n+2}{n-2} \  (n\geq3)$. If $u_0\in C(\bar\Omega)\cap H_0^1(\Omega)$ and $J(u_0)=d,I(u_0)<0$, then the solution of problem (\ref{eq:pro}) blows up in finite time. 
\end{Thm}

After that, for the low initial energy and critical initial energy solution of  (\ref{eq:pro}) i.e. $J(u_0)\leq d$, we study the asymptotic behavior.
\begin{Thm}
	\label{thm:4.1}
	Let $\Omega$ be a smooth bounded convex domain in  $\mathbb{R}^n \ (n=3\,\textrm{or}\,4)$. Assume $1<p<\frac{n+2}{n-2}$, such that $(p-1)(2-\frac{1}{p})<\frac{4}{n-2}$. If $u_0\in C(\bar \Omega)\cap H_0^1(\Omega)$ satisfies $J(u_0)\leq  d$ and $I(u_0)>0$, then for the global solution $u(t)$ of problem (\ref{eq:pro}) decays to $0$ exponentially as $t\to \infty$.
\end{Thm}
 \begin{Thm}
	\label{thm:4.3}
	Let $\Omega$ be a smooth bounded convex domain and $1<p<\frac{n+2}{n-2} \  (n\geq 3)$. If $u_0\in C(\bar\Omega)\cap H_0^1(\Omega)$ satisfies $J(u_0)\leq  d$ and $I(u_0)<0$, then the corresponding solution of problem (\ref{eq:pro}) grows as an exponential function in $L^{\frac{2n}{n-2}}(\Omega)$ norm.
\end{Thm}

The reminder of this paper is organized as follows. In the next section, we give some preliminaries about the family of potential wells, after which we discuss the vacuum isolating of solutions for (\ref{eq:pro}). In Section \ref{sec:3}, we establish the threshold for global solutions and finite time blow up solutions of (\ref{eq:pro}) at the critical initial energy level. At last, the asymptotic behavior will be discussed in Section \ref{sec:4}.

Throughout the paper, we denote $v(u)=\frac{1}{|x|^{n-2}}*|u|^p$, $||\cdot ||_p=||\cdot ||_{L^p(\Omega)}$, $||\cdot ||=||\cdot ||_2$ and denote the maximal existence time by $T_{max}$. 

\section{Vacuum Isolating}\label{sec:2}
In this section, we shall introduce a family of Nehari functionals $I_\delta(u)$ in spcace $H_0^1(\Omega)$ and give the corresponding lemmas, which will help us to demonstrate  the vacuum isolating behavior of (\ref{eq:pro}).

\begin{Lem}
\label{lem:2.1}
Let $1<p<\frac{n+2}{n-2}$ and $\delta>0$. Then is a contant $C_{n, p,\Omega, \delta}>0$ satisfies that $||\nabla u||\geq C_{n, p,\Omega, \delta}$ for all $u\in H_0^1(\Omega)\backslash\{0\}$ and $I_\delta(u)\leq 0$.
\end{Lem}
\begin{proof}
 Provided that $I_\delta(u)\leq 0$, applying the classical Hardy-Littlewood-Sobolev inequality we have
\begin{equation}
\label{eq:2.1}
\delta ||\nabla u||^2\leq \int_{\Omega}v(u)|u|^{p}dx=\int_{\Omega\times \Omega}\frac{|u(y)|^p|u(x)|^p}{|x-y|^{n-2}}dxdy \leq C||u||_{2np/(n+2)}^{2p}.
\end{equation}
Notice that $\frac{2np}{n+2}<\frac{2n}{n-2}$ due to $1<p<\frac{n+2}{n-2}$. By using H\"older inequality and Sobolev inequality we obtain
\begin{equation}
\label{eq:2.2}
||u||_{2np/(n+2)}\leq C_{n,p,\Omega}||u||_{2n/(n-2)}\leq C_{n,p,\Omega}||\nabla u||.
\end{equation}
Combining (\ref{eq:2.1}) and (\ref{eq:2.2}), one has $\delta ||\nabla u||^2 \leq C_{n,p,\Omega}||\nabla u||^{2p}$ i.e. $||\nabla u||^2\geq (\delta/C_{n,p,\Omega})^{\frac{1}{2p-2}}= C_{n, p, \Omega, \delta} $.
\end{proof}

\begin{Lem}
\label{lem:2.2} 
Let 
\begin{equation}
\label{eq:2.3}
C^{*}=\sup_{u\in H_0^1(\Omega),||\nabla u||\neq 0}\frac{ \int_\Omega v(u)|u|^pdx}{||\nabla u||^{2p}}.
\end{equation}
Then
\begin{equation}
\label{eq:2.4}
d(\delta) = \inf_{u \in N_\delta}J(u)=\left( \frac{1}{2}\delta^{\frac{1}{p-1}}-\frac{1}{2p}\delta^{\frac{p}{p-1}} \right){C^*}^{-\frac{1}{p - 1}}.
\end{equation}
\end{Lem}

\begin{proof}
At first, by the proof of Lemma  \ref{lem:2.1}, there is $C_{n,p,\Omega}>0$ such that $\int_\Omega v(u)|u|^pdx \leq C _{n,p,\Omega}||\nabla u||^{2p}$ for all $u\in H_0^1(\Omega)\backslash \{0\}$, which ensures the existence of $C^*$.

For each $u\in H_0^1(\Omega)\backslash \{0\}$, there is a unique $\ld=\ld(\delta,u)$ so that $\ld (\delta, u)u\in N_\delta$. A simple calculation gives
\begin{equation*}
\label{eq:lamd}
\ld (\delta)=\left(\frac{\delta ||\nabla u||^2}{\int_\Omega v(u)|u|^pdx}\right)^{\frac{1}{2p-2}},
\end{equation*}
and
 $$J(\ld u)=\left( \frac{1}{2}\delta^{\frac{1}{p-1}}-\frac{1}{2p}\delta^{\frac{p}{p-1}} \right)\left(\frac{||\nabla u||^{2p}}{ \int_\Omega v(u)|u|^pdx}\right)^{\frac{1}{p-1}}.$$
Noticing that
$\inf_{u\in N_{\delta}}J(u)=\inf_{u\in H_0^1(\Omega)\backslash \{0\}} J(\lambda(\delta,u)u)$ and by using the definition of $d(\delta)$ we conclude that
 $$d(\delta) =\inf_{u \in N_\delta}  J(u)=\inf_{u \in H_0^1(\Omega)\backslash \{0\}} J(\ld(\delta, u) u)=\left( \frac{1}{2}\delta^{\frac{1}{p-1}}-\frac{1}{2p}\delta^{\frac{p}{p-1}} \right){C^*}^{-\frac{1}{p - 1}}.$$
\end{proof}

\begin{Lem}
\label{lem:2.3}
$d(\delta)$ satisfies the following properties:
\begin{enumerate}
	\item[(i)] $d(\delta)>0$ for $0<\delta<p$;
	\item[(ii)] $\lim_{\delta \to 0}d(\delta)=\lim_{\delta \to p}d(\delta)=0$;
	\item[(iii)]  $d(\delta)$ is strictly increasing on $0<\delta \leq 1$, strictly decreasing on $1\leq \delta<p$ and takes the maximum $d=d(1)$ at $\delta =1$;
	\item[(iv)] $d(\delta)$ is continuous on $0\leq \delta\leq p$.
\end{enumerate}
\end{Lem}
\begin{proof}
From (\ref{eq:2.4}), $d(\delta) =\frac{1}{2}\delta^{\frac{1}{p-1}}{C^*}^{\frac{-1}{p-1}}\left(1-\frac{\delta}{p}\right)$
		which gives (i)(ii)(iv). By a straightforward calculation, we can verify $d'(1)=0$, $d'(\delta)>0$ for $0<\delta<1$ and $d'(\delta)<0$ for $1<\delta<p$, which shows (iii).

\end{proof}

\begin{Rk}
	From the above Lemma, we know that the depth of the potential well is $d=(\frac{1}{2}-\frac{1}{2p}){C^*}^{-\frac{1}{p - 1}}=\max_{\delta\in [0,p]}d(\delta)$.
\end{Rk}

 \begin{Lem}
 \label{lem:2.5}
Let $0<e<d$ and $\delta_1,\delta_2$ are two roots of equation $d(\delta)=e$. If $u\in H_0^1(\Omega)$ and $J(u)\leq  e$, then the sign of $I_\delta(u)$ remain unchanged on $(\delta_1,\delta_2)$.
 \end{Lem}
\begin{proof}

Assume $I_\delta(u)$ change its sign on $(\delta_1,\delta_2)$, then there exists a $\delta_0$ such that $I_{\delta_0}(u)=0$, that is to say $u\in N_{\delta_0}$ and hence  $J(u)\geq d(\delta_0)$. By using Lemma 2.3, we have $J(u)\geq d(\delta_0)>d(\delta_1)=d(\delta_2)$, which contradicts to the choice of $\delta_1$ and $\delta_2$.
	
\end{proof}


 
 

We are now in a position to give the proof of Theorem \ref{thm:vacuum}.

\begin{proof}[Proof of Theorem \ref{thm:vacuum}]
Let $u(t)\ (0\leq t<T_{max})$ be the solution of problem (\ref{eq:pro}) corresponding to $u_0$.
We only need to prove that if $u_0\neq 0$ and $J(u_0)\leq e$, then for all $\delta\in (\delta_1,\delta_2)$, $u(t)\not\in N_\delta $, i.e. $I_{\delta}(u(t))\neq 0$, for all $t\in [0,T_{\max})$. 

At first, it is clear that $I_{\delta}(u_0)\neq 0$. Since if  $I_{\delta}(u_0)=0$, then $J(u_0)\geq d(\delta)>d(\delta_1)=d(\delta_2)$, which contradicts with the definition of $\delta_1$ and $\delta_2$.

 Suppose there is $t_1>0$ s.t. $u(t_1)\in U_e$. Namely, there is some $\delta\in(\delta_1,\delta_2)$ such that $u(t_1)\in N_\delta$. Since the energy functional $J(u)$ is no increasing along the flow generated by (\ref{eq:pro}), see (\ref{eq:decj}). Thus, we get $J(u_0)\geq J(u(t_1))\geq d(\delta)>J(u_0)$, which leads to a contradiction. 
\end{proof}

\section{Threshold for solutions with critical initial energy}\label{sec:3}

In this section, we deal with the critical initial energy solution.

 \begin{proof}[Proof of Theorem \ref{thm:3.1}]
We may assume that $u(t)\neq 0$ for all $t\in [0,T_{max})$. Actually, if there is 
$u(t) = 0$, then by uniqueness, $u(s) = 0$ for all $s\geq t$. Hence, the conclusion is true.

We claim that $I(u(t))>0$ for any $t\in [0,T_{\max})$. Otherwise, suppose there exists a $t_0>0$ such that $I(u(t_0))=0$, and $I(u(t))>0$ for $0<t<t_0$. Then 
\begin{equation}
	\label{eq:fc}
	J(u(t_0))\geq d=J(u_0)
\end{equation} 
due to $u(t_0)\in N$. 
On the other hand, since $I(u(t))>0$ for $0<t<t_0$ and by using the fact that $\int_{\Omega}uu_tdx=-I(u(t))$, we obtain $u_t\neq 0$ on $(0,t_0)$, which indicates $\int_0^{t_0}||u_t||^2d\tau>0$. 
Integrating equation (\ref{eq:decj}) on interval $(0,t)$, one has
\begin{equation*}
J(u(t))=J(u_0)-\int_0^t||u_t||^2d\tau<J(u_0)=d,
\end{equation*}
which contradicts to (\ref{eq:fc}).

So we have
$$
d=J(u_0)\geq J(u(t))\geq \frac{1}{2}(1-\frac{1}{p})\int_{\Omega}|\nabla u|^2dx,
$$
which indicates that
$$
\int_{\Omega}|\nabla u|^2dx\leq \frac{2p}{p-1}d.
$$
Therefore, $||u(t)||_{H_0^1(\Omega)}$ is uniformly bounded. For those $q$ satisfies $n-1\leq q\leq \frac{2n}{n-2}$ ($n=3$ or $4$), and  $\frac{n}{2}(p-1)(2-\frac{1}{p})<q<\frac{2n}{n-2}$, we have $||u(t)||_{L^q(\Omega)}$ is bounded, by using the Sobolev inequality.  Applying Theorem 6 in \cite{LiuB2014}, we know that $T_{\max}=\infty$.
 
 \end{proof}

 We shall prove Theorem \ref{thm:3.2} by using the concavity method \cite{Levine}.

 \begin{proof}[Proof of Theorem \ref{thm:3.2}]

First we prove $I(u(t))<0$ for $t\in(0,T_{max})$. Suppose it is false, then there exists a $t_0>0$ s.t. $I(u(t_0))=0$ and $I(u(t))=I_1(u(t))<0$ for $0\leq  t<t_0$.  On the one hand we have $||\nabla u(t)||\geq C_{n,p,\Omega}$ on $[0,t_0)$ by using Lemma \ref{lem:2.1}, which implies $u(t_0)\neq 0$. Thus we obtain 
\begin{equation}
\label{eq:j}
J(u(t_0))\geq d
\end{equation}
due to the fact that $u(t_0)\in N$. On the other hand, one can see $u_t\neq 0$ on $(0,t_0)$ since $\int_{\Omega}uu_tdx=-I(u)>0$, which indicates $\int_0^{t_0}||u_t||^2d\tau>0$. By a similar argument as in the  proof of Theorem 1.3, we have $J(u(t_0))=J(u_0)-\int_0^{t_0}||u_t||^2d\tau<d$, which contradicts with (\ref{eq:j}). Consequently, we have  
\begin{equation}
\label{eq:negi}
I(u(t))<0,\quad \forall t\in [0,T_{max}).
\end{equation} 
Moreover, by Lemma  \ref{lem:2.1}, there holds 
\begin{equation}
\label{eq:gradu2}
||\nabla u(t)||\geq C_{n,p,\Omega}, \quad \forall t\in [0,T_{\max}).
\end{equation}

Assume for contradiction that $T_{max}=+\infty$. Denote $M(t)=\frac12\int_0^t||u(\tau)||^2d\tau$. Then we obtain $M'(t)=\frac12||u(t)||^2>0$  and $M''(t)=\int_{\Omega}uu_tdx=-I(u(t))>0$ for $t>0$. Choose $t_1>0$ such that $$0<d_1=J(u(t_1))=J(u_0)-\int_0^{t_1}||u_t||^2d\tau=d-\int_0^{t_1}||u_t||^2d\tau<d.$$
Thus, we have $J(u(t))\leq  d_1$ for each $t\geq t_1$. It follows from (\ref{eq:negi}), Lemma \ref{lem:2.1} and Lemma \ref{lem:2.5} that $I_\delta(u(t))<0$ for  $\delta_1<\delta<\delta_2,t\geq t_1$,  where $\delta_1,\delta_2$ are two roots of equation $d(\delta)=d_1$. Thus, choosing any $\delta_0\in (1,\delta_2)$, we have $I_{\delta_0}(u(t))<0$ for all $t\geq t_1$. Taking (\ref{eq:gradu2}) into account, we find
$$M''(t)=-I(u(t))=(\delta_0-1)||\nabla u(t)||^2-I_{\delta_0}(u(t))>(\delta_0-1)C_{n,p,\Omega}>0, \forall t\geq t_1,$$
which indicates $M'(t)\to +\infty$ as $t\to+\infty$ and $M(t)\to +\infty$ as $t\to+\infty$.

Now for $t>0$, we estimate the following
\begin{align}
 M''(t)=-I(u(t))&=(p-1)||\nabla u(t)||^2-2pJ(u(t))\label{eq:3.2}\\
 &\geq 2p\int_0^t||u_t||^2d\tau+(p-1)\lambda M'(t)-2pJ(u_0) \label{eq:3.3},
\end{align}
here constant $\lambda$ satisfies $||\nabla u||^2\geq \frac{\lambda}{2}||u||^2$ which from Poincar\'e inequality. Integrating $M''(t)=\int_{\Omega}uu_tdx$ on $(0,t)$ yields
$$M'(t)-M'(0)=\int_0^t\int_{\Omega}uu_tdxd\tau.$$
Hence,
\begin{align}
(M'(t))^2&=-(M'(0))^2+2M'(t)M'(0)+\left(\int_0^t\int_{\Omega}uu_tdxd\tau\right)^2 \notag \\
&=-\frac14||u_0||^4+M'(t)||u_0||^2+\left(\int_0^t\int_{\Omega}uu_tdxd\tau\right)^2\notag \\
&\leq  M'(t)||u_0||^2+\left(\int_0^t\int_{\Omega}uu_tdxd\tau\right)^2 \label{eq:3.4}.
\end{align}
Then combining (\ref{eq:3.3}) and (\ref{eq:3.4}), we have
\begin{align}
 MM''-pM'^2 &\geq p\left[ \int_0^t||u||^2d\tau \cdot \int_0^t||u_t||^2d\tau -\left(\int_0^t\int_{\Omega}uu_tdxd\tau\right)^2\right] \notag \\
 &+(p-1)\lambda MM'-2pMJ(u_0)-pM'||u_0||^2 \notag \\
 &\geq (p-1)\lambda MM'-2pMJ(u_0)-pM'||u_0||^2 \label{eq:3.5},
 \end{align}
 where we have used Schwatz's inequality.
 Since $M(t)\to\infty$ and $M'(t)\to \infty$ as $t\to\infty$, then there exists a $t_2$ s.t. 
 $$ \frac{p-1}{2}\lambda M(t)>p||u_0||^2 ,\ \frac{p-1}{2}\lambda M'(t)>2pJ(u_0) , \, t>t_2.$$
 Hence we obtain by (\ref{eq:3.5})
 $$ M(t)M''(t)-pM'(t)^2>0, \, t>t_2.$$
Let us consider the function $M^{-p+1}(t)$. By a simple calculation we have
$$
 \frac{d^2}{dt^2}M^{-p+1}(t)=(-p+1) M^{-p-1}(t)\left(M(t)M''(t)-pM'(t)^2\right)<0 , t>t_2.
$$
 It guarantees that nonincreasing function $M^{-p+1}(t)$ is concave on $(t_2,\infty)$. Consequently, there exists a finite time $T>0$ such that $\lim_{t\to T}M^{-p+1}(t)=0$ i.e. $\lim_{t\to T}M(t)=\infty$ which contradicts the assumption that $T_{\max}=+\infty$.
 
This completes the proof. 

 \end{proof}
  
  We conclude this section by pointing out the following remark.
 \begin{Rk}
 \label{rk:3.3}
 From the proof of Theorems \ref{thm:3.1} and \ref{thm:3.2}, we can see that 
 \begin{gather*}
 W^{'}=\left\{u \in H_0^1(\Omega)|J(u)\leq  d,I(u)>0\right \}\cup \{ 0 \}\ \textrm{and}\\ 
 Z^{'}=\left\{u \in H_0^1(\Omega)|J(u)\leq  d,I(u)<0\right\},
 \end{gather*}
 are both invariant for solutions of problem (\ref{eq:pro}). Moreover, the solution has long time existence if $u_0\in W^{'}$ and the solution blows up at finite time if $u_0\in  Z^{'}$.
 \end{Rk}
  
\section{Exponential decay, exponential growth}\label{sec:4}
In this section, we shall investigate the asymptotic behavior of solutions for problem (\ref{eq:pro}) with $J(u_0)\leq  d$ and give the proof of Theorem \ref{thm:4.1} and Theorem \ref{thm:4.3}.

\begin{proof}[Proof of Theorem \ref{thm:4.1}]
We consider the following two cases. 

$\mathbf{Case \,1}$. $J(u_0)<d$.

By using
$$J(u)=\left(\frac12-\frac{1}{2p}\right)||\nabla u||^2+\frac{1}{2p}I(u),$$
we get $J(u_0)>0$. Let $\delta_1,\delta_2$ $(\delta_1<\delta_2)$ be the two roots of equation $d(\delta)=J(u_0)$.

From Proposition 10 in \cite{LiuB2014}, we know  $J(u(t))<d,I(u(t))>0$ for all $t>0$, provided $J(u_0)<d,I(u_0)>0$.  Since $J(u(t))\leq  J(u_0)<d, I(u(t))>0$ for each $t\geq 0$, we obtain $I_{\delta}(u(t))> 0$ for $\delta\in(\delta_1,\delta_2),t\geq 0$ by using Lemma \ref{lem:2.5}. Taking any $\delta_0\in(\delta_1,1)$, we have
$$\frac12\frac{d}{dt}||u||^2+I(u)=\frac12\frac{d}{dt}||u||^2+(1-\delta_0)||\nabla u||^2+I_{\delta_0}(u)=0.$$
 By applying Poincar\'e inequality, we obtain
$$\frac12\frac{d}{dt}||u||^2+(1-\delta_0)C||u||^2< 0 ,\quad t \geq 0.$$
Consequently, by using Gronwall inequality we know that
$$ ||u(t)||^2\leq || u_0||^2e^{-2C(1-\delta_0)t},\quad 0\leq  t<\infty.$$

$\mathbf{Case \,2}$. $J(u_0)=d$.

Indeed, given $J(u_0)=d,I(u_0)>0$, from the proof of Theorem \ref{thm:3.1} we can choose any fixed $t_0>0$ such that $0<J(u(t_0))<d,I(u(t_0))>0$. Let $\delta_1,\delta_2$ are two roots of equation $d(\delta)=J(u(t_0))$. Thus by a similar argument with proof of $\mathbf{Case \,1}$, we easily obtain 
$$ ||u(t)||^2\leq || u(t_0)||^2e^{-2C(1-\delta_0)t},\quad t_0\leq  t< \infty.$$
Therefore the result of theorem follows immediately.
\end{proof}

In order to prove Theorem \ref{thm:4.3}, we need the following lemma.
 \begin{Lem}
 \label{lem:4.2}
 Let $u(t)$ be a nontrival solution of problem (\ref{eq:pro}) which satisfies $J(u_0)<d, ||\nabla u_0||>\alpha_1:={C^*}^{\frac{-1}{2p-2}}$. Then there exists $\alpha_2>\alpha_1$ such that $||\nabla u(t)||\geq \alpha_2$ for all $0\leq t<T_{max}$. Here $C^*$ is defined by (\ref{eq:2.3}).
 \end{Lem}
 
 \begin{proof}
 Firstly, by the definition of $C^*$ as in (\ref{eq:2.3}), we estimate
 \begin{equation}
 \label{eq:4.1}
 J(u)=\frac12||\nabla u||^2-\frac{1}{2p}\int_{\Omega}v(u)|u|^pdx\geq\frac12||\nabla u||^2-\frac{C^*}{2p}||\nabla u||^{2p}.
 \end{equation}
 Denote $g(\alpha)=\frac12\alpha^2-\frac{C^*}{2p}\alpha^{2p}$. It's easy to verify $g(\alpha)$ is strictly increasing on $(0,\alpha_1)$, decreasing on $(\alpha_1,\infty)$ and attains its maximum at $\alpha=\alpha_1$:
 \begin{equation}
 \label{eq:4.2}
 g(\alpha_1)=\left(\frac12-\frac{1}{2p}\right){C^{*}}^{-1/(p-1)}=d.
 \end{equation}
 Let $\alpha_0=||\nabla u_0||>\alpha_1$, we obtain $g(\alpha_0)=g(||\nabla u_0||)\leq J(u_0)<d$ by  formula (\ref{eq:4.1}). Hence, we can find a $\alpha_2\in (\alpha_1,\alpha_0]$ such that $g(\alpha_2)=J(u_0)$. 
 
 We claim that $||\nabla u(t)||\geq \alpha_2$ for all $t\geq 0$. 
 Otherwise, by the continuity, we can choose $t_0>0$ such that $\alpha_1<||\nabla u(t_0)||< \alpha_2$. 
 Thus we know that  $d=g(\alpha_1)>g(||\nabla u(t_0)||)>g(\alpha_2)=J(u_0)$, which contradicts with the fact that 
 $J(u_0)<d$.
 
 The proof is now complete. 
 \end{proof}
 
With the help of the above lemma, we give the proof of Theorem \ref{thm:4.3}.
 
 \begin{proof}[Proof of Theorem \ref{thm:4.3}]
Let us consider the following two cases. 

  $\mathbf{Case \,1}$. $J(u_0)<d$.
  
On the one hand, since $I(u_0)<0$ and by using (\ref{eq:2.3}) we obtain 
  $$||\nabla u_0||^2<\int_{\Omega}v(u_0)|u_0|^pdx\leq{C^*}||\nabla u_0||^{2p},$$
which implies $||\nabla u_0||>{C^{*}}^{-1/(2p-2)}=\alpha_1$. Applying  Lemma \ref{lem:4.2} we get 
\begin{equation}
\label{eq:gradu}
||\nabla u(t)||\geq \alpha_2>\alpha_1, \quad t\in [0, T_{\max}).
\end{equation}

For $t\geq 0$, define $H(t)=d-J(u(t))$, $L(t)=H(t)+\frac12||u(t)||^2$.  Combining (\ref{eq:decj}), (\ref{eq:3.2}) and (\ref{eq:gradu}), we have
\begin{equation}
\label{eq:H}
H(t)>0
\end{equation}
and
\begin{align*}
L'(t)&=-\frac{dJ(u(t))}{dt}-I(u(t))=||u_t||^2+(p-1)||\nabla u||^2+2pH(t)-2pd \label{eq:4.6}\\
&\geq(p-1)||\nabla u||^2+2pH(t)-2pd\\
&=(p-1)\frac{\alpha_2^2-\alpha_1^2}{\alpha_2^2}||\nabla u||^2+(p-1)\frac{\alpha_1^2}{\alpha_2^2}||\nabla u||^2+2pH(t)-2pd\\
&\geq(p-1)\frac{\alpha_2^2-\alpha_1^2}{\alpha_2^2}||\nabla u||^2+(p-1)\alpha_1^2+2pH(t)-2pd .
\end{align*}
Notice that $2pd=(p-1){C^{*}}^{-1/(p-1)}=(p-1)\alpha_1^2$ follows from formula (\ref{eq:4.2}), then
\begin{equation}
\label{eq:4.3}
L'(t)\geq (p-1)\frac{\alpha_2^2-\alpha_1^2}{\alpha_2^2}||\nabla u||^2+2pH(t).
\end{equation}
 Let $C_1=\min \left \{2p,(p-1)\frac{\alpha_2^2-\alpha_1^2}{\alpha_2^2}\right \}$. Taking (\ref{eq:H}) into account, we have
\begin{equation}
\label{eq:4.4}
L'(t)\geq C_1\left(H(t)+||\nabla u||^2\right).
\end{equation}
Applying Poincar\'e inequality, one has
\begin{align}
L(t)=H(t)+\frac12||u||^2&\leq H(t)+\frac{C}{2}||\nabla u||^2 \notag\\
&\leq C_2\left(H(t)+||\nabla u||^2\right), \label{eq:4.5}
\end{align}
here $C_2=\max \left\{1, \frac{C}{2}\right\}$. Combining (\ref{eq:4.4}) and (\ref{eq:4.5}) we find that there exists $C_3>0$ such that $L'(t)\geq C_3L(t)$ for $t\geq 0$. Consequently, by Gronwall inequality we obtain
\begin{equation}
\label{eq:4.6}
L(t)\geq L(0)e^{C_3t},\quad t\geq 0.
\end{equation}

On the other hand, from formulas (\ref{eq:2.1}) and (\ref{eq:2.2}), we get $||\nabla u||^2\leq ||u||_{2n/(n-2)}^{2p}$ by taking $\delta=1$. By using Poincar\'e inequality, we find $||u||^2\leq C ||\nabla u||^2\leq C||u||_{2n/(n-2)}^{2p}$. Hence, combining the above estimates and (\ref{eq:2.1}), we obtain
\begin{align}
L(t)&=d+\frac{1}{2p}\int_{\Omega}v(u)|u|^pdx -\frac12||\nabla u||^2+\frac12||u||^2 \notag\\
&\leq d+C||u||_{2n/(n-2)}^{2p}  \label{eq:4.7}.
\end{align} 

Therefore, combining (\ref{eq:4.6}) and (\ref{eq:4.7}), it follows that $||u(t)||_{2n/(n-2)}$ will increase as an exponential function.
  
  $\mathbf{Case \,2}$. $J(u_0)=d$.
  
  Given an any fixed $t_0>0$, from the proof of Theorem \ref{thm:3.2}, we know $J(u(t)<d,I(u(t))<0$ for $t\geq t_0$. We also define $H(t)=d-J(u(t))>0,L(t)=H(t)+\frac12||u(t)||^2$ for $t\geq t_0$. Thus proceeding as in the proof of  $\mathbf{Case \,1}$, we see that the theorem holds.

This completes the proof.
  
 \end{proof}

\end{document}